\documentclass[reqno,12pt]{amsart}
\usepackage{amsmath, amssymb, amsthm, epsfig}

\usepackage{color}
\usepackage[inner=1.0in,outer=1.0in,bottom=1.0in, top=1.0in]{geometry}

\usepackage{hyperref, latexsym}
\usepackage{url}

\usepackage{color}

\def\today{\ifcase\month\or
  January\or February\or March\or April\or May\or June\or
  July\or August\or September\or October\or November\or December\fi
  \space\number\day, \number\year}

 \newtheorem{theorem}{Theorem}[section]

 \newtheorem{lemma}[theorem]{Lemma}

 \newtheorem{corollary}[theorem]{Corollary}
 \theoremstyle{definition}

 \theoremstyle{remark}

\newcommand{\meas}{\textup{meas}}

\def\E{\mathbb E}

\def\P{\mathbb P}

\def\n{\mathbf n}

\newcommand{\re}{\textup{Re}}
\newcommand{\im}{\textup{Im}}

 \def\Phrand{ \Phi_{\mathrm{rand}}}

 \def\Phrhat{\widehat{\Phi}_{\mathrm{rand}}}

\makeatletter
\@namedef{subjclassname@2010}{  \textup{2010} Mathematics Subject Classification}
\makeatother

\begin{document}

\title[An asymptotic expansion of Selberg's central limit theorem]{An asymptotic expansion of Selberg's central limit theorem  near the critical line} 

\date{\today}

\author[Y. Lee]{Yoonbok Lee}
\address{Department of Mathematics,  Research Institute of Basic Science,   Incheon National University,   119 Academy-ro, Yeonsu-gu, Incheon,  22012, Korea}
\email{leeyb@inu.ac.kr, leeyb131@gmail.com}

\allowdisplaybreaks
\numberwithin{equation}{section}

\keywords{Selberg's central limit theorem, Riemann zeta function}

\subjclass[2010]{11M06}


\maketitle

\begin{abstract}
We find an asymptotic expansion of Selberg's central limit theorem for the Riemann zeta function on $\sigma  = \frac12 + ( \log T)^{-\theta}$ and $t \in [T, 2T]$, where $ 0 < \theta < \frac12$ is a constant.
\end{abstract}

\section{Introduction}

Let $ \theta > 0 $ and $ \sigma_T := \sigma_T(\theta) = \frac12+ ( \log T)^{-\theta}$ throughout the paper. Selberg's central limit theorem (Theorem 2 in \cite{Sel}) says that for $ \frac12 \leq \sigma \leq \sigma_T$,  the function 
$$ \frac{ \log \zeta( \sigma+ it ) }{ \sqrt{ \pi  \sum_{ p < t } p^{- 2 \sigma}}} $$
 has a normal Gaussian distribution in the complex plane. Note that 
 $$\sum_{ p < t } p^{- 2 \sigma} = \log \bigg( \min \bigg(  \frac{1}{\sigma- 1/2} ,   \log t \bigg)\bigg) +O(1) $$
  for $ \sigma \geq \frac12 $. 
 Recently, Radziwi\l\l~ and Soundararajan in \cite{RS} provide a simple proof of Selberg's central limit theorem for $\log | \zeta ( \frac12 + it ) | $.

When $ \sigma > 1/2 $ is not too close to $ 1/2$, $ \log \zeta ( \sigma + i t ) $ has a nice approximation by a Dirichlet polynomial so that allows us to improve Selberg's central limit theorem by  finding lower order terms.
 In this direction, Ha and Lee in \cite{HL} prove the following theorem.
\begin{theorem}[Corollary 1.5 of \cite{HL}]
Let $0<\theta<\frac12 $, $a<b$ and $c<d$ be real numbers. There exist polynomials $ g_k ( x,y)$ of degree $\leq k $ such that
\begin{equation}\label{eqn:cor1.5:HL}\begin{split}
\frac1T \meas& \{ t \in [T,2T] :  \frac{\log\zeta(\sigma_T+it )}{\sqrt{ \pi \psi_T } } \in[a,b] \times [c,d] \}  \\ 
&=  \sum_{ 0 \leq k  \leq 5  } \frac{ 1  }{ \sqrt{\psi_T}^{k}}   \int_c^d \int_a^b g_k ( x,y )  e^{- \pi(x^{2}+y^{2}) } dx dy + O \bigg(  \frac{1}{ ( \log \log T)^3}\bigg)
\end{split}\end{equation}
as $ T \to \infty$, where $g_0(x,y)=1 $ and 
\begin{equation}\label{def psi T} 
   \psi_T  :=\sum_{p  } \sum_{k \geq 1 } \frac{1}{k^{2}}p^{-2k\sigma_T }=\theta\log\log T+O(1).
   \end{equation}
\end{theorem}
Our main theorem improves it by finding lower order terms and expressing the integral of \eqref{eqn:cor1.5:HL} in terms of Hermite polynomials.
\begin{theorem}\label{main theorem}
Let $0<\theta< \frac12 $, $a<b$ and $c<d$ be real numbers. There exist constants $ \epsilon, \eta>0$ and a sequence $\{ d_{k,\ell}\}_{k,\ell \geq 0 } $ of real numbers  such that 
\begin{align*}
\frac1T \meas & \{ t \in [T,2T] :  \frac{\log\zeta(\sigma_T+it )}{\sqrt{ \pi \psi_T } } \in[a,b] \times [c,d] \} \\
&=  \sum_{ k+ \ell \leq \epsilon \psi_T   } \frac{d_{k,\ell} }{      \sqrt{\psi_T} ^{k+\ell}} \int_{c }^{d } \int_{a  }^{b }   e^{ - \pi ( x^2 + y^2 ) }   H_k (\sqrt{ \pi  } x)H_\ell(\sqrt{ \pi  } y)  dx dy +    O \bigg( \frac{ 1}{ ( \log T)^{\eta }} \bigg)
\end{align*}
  as $ T \to \infty$, where $H_n (x)$ is the $n$-th Hermite polynomial defined by 
\begin{equation}\label{def Hermite poly}
 H_n (x) :=  (-1)^n e^{x^2} \frac{d^n}{dx^n} ( e^{-x^2}). 
 \end{equation}
 Moreover, $d_{0,0}=1$ and $ d_{k, \ell}=0 $ for $ k+\ell = 1,2$. 
\end{theorem}
 The sequence $\{ d_{k,\ell} \} $ is defined by its generating function in \eqref{def dkl}.
Since $d_{k,\ell}=0$  for $ k+\ell = 1, 2$ in Theorem \ref{main theorem}, we have  the following corollary.
\begin{corollary} 
Let $0<\theta< \frac12 $, $a<b$ and $c<d$ be real numbers, then we have
$$ \frac1T \meas   \{ t \in [T,2T] :  \frac{\log\zeta(\sigma_T+it )}{\sqrt{ \pi \psi_T } } \in[a,b] \times [c,d] \}  = \int_{c }^{d } \int_{a  }^{b }   e^{ - \pi ( x^2 + y^2 ) }     dx dy  + O \bigg( \frac{ 1}{ ( \log \log T)^{3/2}} \bigg) .$$
\end{corollary}

Remark that Hejhal in \cite[Theorem 2.1]{H} extends Theorem \ref{main theorem} to a multi-dimensional setting. For a verification, we provide a useful identity  
\begin{align*}
\sqrt{ \pi} \int_{x_1}^{x_2} e^{-\pi x^2} H_{n+1}  (\sqrt{\pi} x)  dx   &= e^{- \pi x_1^2 } H_{n}(\sqrt{\pi}x_1) - e^{-\pi x_2^2 } H_{n}(\sqrt{\pi}x_2 ) \\
&=  (- \sqrt{ \pi})^{-n} (\Phi^{(n+1)} ( x_1 ) -\Phi^{(n+1)} ( x_2 )      )   
\end{align*}
for $ n \geq 0 $, where $\Phi ( x) : =  \int_0^x e^{- \pi u^2 }du $ is defined in \cite{H}. 
However, this paper is still meaningful, since there is only a sketched proof in \cite{H} and the asymptotic expansion in our theorem is of length $ \sim \epsilon \theta  \log \log T$, 
while the expansion in \cite{H} is of any constant length. 

We also remark that the expansion in Theorem \ref{main theorem} is similar to an Edgeworth expansion in the probability theory. For further information, see Chapter 7 of \cite{C}.

We prove Theorem \ref{main theorem} at the end of Section \ref{section 2}.

\section{Estimates on the random model}\label{section 2}

The random Riemann zeta function is defined by the product
\begin{equation}\label{def zeta X}
 \zeta( \sigma , X)  : = \prod_p  \bigg( 1 - \frac{ X(p)}{p^\sigma} \bigg)^{-1} ,
 \end{equation}
 where $X(p)$ is independent and identically distributed random variables on the unit circle $|z|=1$ assigned for each prime $p$. The product converges almost surely for $ \sigma > \frac12$. 
If $ \sigma  > \frac12  $ is not too close to $ \frac12$, then  the distribution of the random model $ \log \zeta( \sigma , X) $  approximates that of   $ \log \zeta(\sigma+it)$. More precisely, the discrepancy defined by
$$ D_{\sigma }(T) := \sup_{\mathcal{R}} \left|  \frac1T \meas \{ t \in [ T, 2T] : \log \zeta( \sigma + it  ) \in \mathcal{R} \} - \P [ \log \zeta ( \sigma ,X ) \in \mathcal{R}  ]   \right|$$
is small for $ \sigma \geq \sigma_T$,  where the supremum is taken over rectangles $\mathcal{R}$ with sides parallel to the coordinate axes.  Lamzouri, Lester and Radziwi\l\l~ in \cite{LLR} show that    
 $$D_{\sigma} (T) = O\bigg( \frac{1}{ ( \log T)^{ \sigma }} \bigg) $$
  holds for fixed $ \sigma > \frac12 $, which improves earlier results of Matsumoto \cite{M1},  \cite{M2} and Matsumoto and Harman \cite{HM}. By the same method, Ha and Lee in \cite{HL} show that  for each $ 0 < \theta <  \frac12 $, there is a constant $\eta >0$ such that
  \begin{equation}\label{eqn disc}
 D_{ \sigma_T} (T) = O_\eta \bigg( \frac{1}{ ( \log T)^{ \eta }} \bigg) .
 \end{equation}

Define 
$$  \Phrand (\mathcal{B} ) :=  \P \left[  \log \zeta ( \sigma_T , X ) \in \mathcal{B} \right] $$
for a Borel set $\mathcal{B} \subset \mathbb{C}$.  It is known that this measure has a density function 
$F_{\sigma_T } $ such that
\begin{equation}\label{eqn prob density}
 \P \left[ \log\zeta(\sigma_T , X)\in\mathcal{B}\right] =\iint_{\mathcal{B}}F_{\sigma_T}(x,y)dxdy  
 \end{equation}
holds for any region $\mathcal{B}$. For a proof, see \cite[Theorem 11]{BJ} or  \cite[Proposition 3.1]{HL}.
Since we have
\begin{equation}\label{eqn clt proof 1}\begin{split}
\frac1T \meas & \{ t \in [T,2T] :  \frac{\log\zeta(\sigma_T+it )}{\sqrt{ \pi \psi_T } } \in[a,b] \times [c,d] \}  \\
&=
\P\left[ \frac{\log\zeta(\sigma_T, X )}{\sqrt{ \pi \psi_T } } \in [a,b] \times [c,d] \right]  + O \bigg(  \frac{1}{ ( \log T)^\eta } \bigg)\\
&=    \int_{c\sqrt{ \pi \psi_T } }^{d\sqrt{ \pi \psi_T } } \int_{a\sqrt{ \pi \psi_T } }^{b\sqrt{ \pi \psi_T } } F_{\sigma_T }( x,y)  dx dy    + O \bigg(  \frac{1}{ ( \log T)^\eta} \bigg)   
\end{split}\end{equation}
by \eqref{eqn disc} and \eqref{eqn prob density}, 
 it is enough to find an asymptotic for $F_{\sigma_T}( x, y ) $ to prove Theorem \ref{main theorem}. 
Since we have
  \begin{equation}\label{F Phrhat int}
   F_{ \sigma_T } ( x,y) = \iint_{\mathbb{R}^2} \Phrhat(u,v) e^{- 2 \pi i ( ux + vy) } du dv  
   \end{equation}
   by the Fourier inversion, we next estimate the Fourier transform
    $$ \Phrhat (u,v) = \E \left[ e^{  2 \pi i ( u \re( \log \zeta( \sigma_T , X) )+  v \im ( \log \zeta( \sigma_T , X) ) } \right]. $$ 
   
   By \eqref{def zeta X} we have
   \begin{equation}\label{eqn phrhat prod}
   \Phrhat (u,v) =  \prod_p J(  \pi u , \pi v, p^{-\sigma_T}) ,
   \end{equation}
where
$$ J( u,v,w ) := \E \left[ e^{ - 2  i ( u \re  \log ( 1 - wX) + v \im \log ( 1- wX)) } \right] . $$
Then we have the following lemma, which  is a modification of \cite[Lemma 3.3]{HL}.
\begin{lemma}\label{lemma1}
Let $ 0 < r   < 1 $ and  $C_r = - \frac{1 }{r}\log ( 1-r)  $. Then we have series expansions 
$$  J ( u,v,w)  = 1 + \sum_{ k , \ell \geq 1 } \frac{ i^{k+\ell}}{k! \ell!} a_{k, \ell} ( w) (u+iv)^k (u-iv)^\ell $$ 
for any $ u, v \in \mathbb{R}$ and $ 0 < w < 1 $, and 
\begin{equation}\label{log J exp}
 \log J ( u,v,w)  = \sum_{ k , \ell \geq 1 } \frac{ i^{k+\ell}}{k! \ell!} b_{k, \ell} ( w) (u+iv)^k (u-iv)^\ell 
 \end{equation}
for $ u^2 + v^2  \leq (2rC_r )^{-2} $ and $ |w| \leq r $, where the coefficients $a_{k,\ell}(w)$ and $b_{k,\ell} ( w )$ are defined by
\begin{align}
 a_{k, \ell}(w) & = \sum_{n \geq \max(k,\ell)} \bigg( \sum_{ \substack{  n_1 +\cdots +n_k = n \\n_i \geq 1  }}  \frac{1}{ n_1 \cdots n_k } \bigg) \bigg( \sum_{ \substack{   m_1 +\cdots +m_\ell = n \\ m_i \geq 1  }} \frac{1}{ m_1 \cdots m_\ell} \bigg) w^{2n} \label{def akl} ,\\
   b_{k,\ell}(w) & =  \sum_{ n  \leq \min(k,\ell)} \frac{(-1)^{n-1}}{n} \sum_{  \substack{ k_1 + \cdots k_n = k \\ \ell_1 + \cdots + \ell_n = \ell  \\ k_i, \ell_i \geq 1 } } {   k \choose { k_1 , \ldots , k_n }} { \ell \choose { \ell_1 , \ldots , \ell_n }} a_{k_1 , \ell_1 } (w) \cdots a_{ k_n , \ell_n }(w)\label{def bkl} .
   \end{align} 
Moreover, we have
\begin{enumerate}
 \item $ b_{k,\ell}(w)$ is real and  $ b_{1,1} ( w ) = \sum_{m \geq 1 } \frac{1}{m^2} w^{2m} $,
 \item $ a_{k,\ell}(w), b_{k, \ell} (w) \ll_{k,\ell} w^{ 2 \max(k,\ell)} $, 
 \item $a_{k,\ell}(w) = a_{\ell, k}(w)  $ and $b_{k,\ell}(w) = b_{\ell, k}(w)  $, 
\item $ 0 < a_{k,\ell}(w)   \leq C_r^{k+\ell} w^{k+\ell} $ and     
$ | b_{k,\ell}(w) | \leq   C_r^{k+\ell} \min(k,\ell)^{k+\ell  } w^{k+\ell} $  for $ 0 < w \leq r $.
\end{enumerate}
\end{lemma}
 \begin{proof}
	The lemma is basically Lemma 3.3 of \cite{HL}. (See \cite[page 852, lines 10 and 20]{HL} for  \eqref{def akl} and \eqref{def bkl}.) The condition $b_{k, \ell} (w ) = b_{\ell, k}(w)$ is the only statement, which is not proved therein. However, this easily follows from \eqref{def bkl} and the fact that $a_{k, \ell} (w ) = a_{\ell, k}(w)$.
 \end{proof}

\begin{lemma}\label{lemma2}
Let $\psi_T$ be as in \eqref{def psi T}. Define 
\begin{equation}\label{def tildebkl}
 \tilde{b}_{k , \ell} := \frac{ (\pi i)^{k+\ell}}{k! \ell!}\sum_p  b_{k, \ell} ( p^{-1/2} ) ,
 \end{equation}
where $ b_{k, \ell}(w)$ is defined in \eqref{def bkl}. 
 Then there is a constant $\delta_1 > 0 $
$$ \sum_p \log J( \pi u , \pi v, p^{- \sigma_T}) = - \pi^2 (u^2+v^2) \psi_T + \sum_{ \substack{ k , \ell \geq 1 \\ k+\ell \geq 3 }} \tilde{b}_{k , \ell} (u+iv)^k (u-iv)^\ell  + O\bigg( \frac{ 1}{ ( \log T)^\theta}   \bigg)$$
for $u^2 + v^2 \leq  \delta_1$.
\end{lemma}
\begin{proof}
Since the inequality  $ 1-e^{-x} \leq x $ holds for any $ x \geq 0 $, we have 
\begin{equation}\label{lemma2 eqn1}
0 <  1 -  p^{-  \frac{ 2n}{ ( \log T)^\theta}} \leq  \frac{ 2n \log p }{ ( \log T)^\theta}  
\end{equation}
for any prime $p$ and integer $n > 0 $.   
For any $ \epsilon>0$, there is a constant $ C(\epsilon)>0$ such that  $ \log x \leq C(\epsilon) x^\epsilon $ for all $ x \geq 1 $.  Thus, we have
\begin{equation}\label{lemma2 eqn2}
 \log p^n   \leq C(2 \epsilon) p^{2n\epsilon } .
 \end{equation}
 By  \eqref{def akl}, \eqref{lemma2 eqn1}, \eqref{lemma2 eqn2}, and (4) of Lemma \ref{lemma1}, we have
\begin{equation}\label{eqn akl 12 sigmaT}\begin{split} 
0<  a_{k,\ell} & ( p^{-1/2}) - a_{k, \ell}  ( p^{-\sigma_T})   \\
& \leq \frac{2 C(2\epsilon)}{( \log T)^\theta} \sum_{n \geq \max(k,\ell)} \bigg( \sum_{ \substack{  n_1 +\cdots +n_k = n \\n_i \geq 1  }}  \frac{1}{ n_1 \cdots n_k } \bigg) \bigg( \sum_{ \substack{   m_1 +\cdots+  m_\ell = n \\ m_i \geq 1  }} \frac{1}{ m_1 \cdots m_\ell} \bigg)   \frac{ 1  }{p^{ (1-2\epsilon)n }  } \\
& = \frac{2 C(2\epsilon)}{( \log T)^\theta} a_{k,\ell} ( p^{-1/2+\epsilon}) \leq  \frac{2 C(2\epsilon)}{( \log T)^\theta}  C_r^{k+\ell} p^{ - (1/2-\epsilon) (k+\ell)}
\end{split}\end{equation}
for $ k, \ell \geq 1 $, any prime $p$ and any $ \epsilon>0 $ with a choice $ r = 2^{-1/2+\epsilon} $. 
Since
$$ \prod_{j=1}^n x_j - \prod_{j=1}^n y_j =  \sum_{j=1}^n  \bigg(  \prod_{i\geq j} x_i \prod_{i<j} y_i - \prod_{i >  j  } x_i \prod_{i \leq j  } y_i\bigg) ,$$
  we have
\begin{align*}
 0< \prod_{j=1}^n & a_{k_j, \ell_j } ( p^{-1/2}) - \prod_{j=1}^n a_{k_j, \ell_j } ( p^{-\sigma_T})    \leq \sum_{j=1}^n (  a_{k_j, \ell_j } ( p^{-1/2}) -   a_{k_j, \ell_j } ( p^{-\sigma_T}) ) \prod_{i \neq j } a_{k_i, \ell_i } ( p^{-1/2})  \\
 & \leq \sum_{j=1}^n  \frac{2 C(2\epsilon)}{( \log T)^\theta}  C_r^{k_j+\ell_j} p^{ - (1/2-\epsilon) (k_j+\ell_j )}    \prod_{i \neq j } C_r ^{k_i + \ell_i}   p^{-1/2(k_i + \ell_i )}    \\
 & \leq   n  \frac{2 C(2\epsilon)}{( \log T)^\theta}  C_r^{\sum_j (k_j+  \ell_j)} p^{ - (1/2-\epsilon) \sum_j  (k_j+\ell_j )}  
\end{align*}
by \eqref{eqn akl 12 sigmaT} and (4) of Lemma \ref{lemma1}. 
The equation   \eqref{def bkl} and the above inequality imply that
\begin{align*}
| b_{k, \ell }& ( p^{-1/2} ) - b_{k, \ell } ( p^{-\sigma_T} ) |  \\
\leq &  \sum_{ n  \leq \min(k,\ell)} \frac{1}{n} \sum_{  \substack{ k_1 + \cdots k_n = k \\ \ell_1 + \cdots + \ell_n = \ell  \\ k_i, \ell_i \geq 1 } } {   k \choose { k_1 , \ldots , k_n }} { \ell \choose { \ell_1 , \ldots , \ell_n }}  \left| \prod_{j=1}^n   a_{k_j, \ell_j } ( p^{-1/2}) - \prod_{j=1}^n a_{k_j, \ell_j } ( p^{-\sigma_T}) \right| \\
\leq &  \sum_{ n  \leq \min(k,\ell)} n^{k+\ell}  \frac{2 C(2\epsilon)}{( \log T)^\theta}  C_r^{k+\ell } p^{ - (1/2-\epsilon)   (k +\ell  )} \\
\leq &   \frac{3 C(2\epsilon)}{( \log T)^\theta}  ( \min(k,\ell))^{k+\ell}   C_r^{k+\ell } p^{ - (1/2-\epsilon)   (k +\ell  )}.
\end{align*}
By Stirling's formula and the above inequality with $ 0 < \epsilon < \frac16$, we have
\begin{multline}\label{lemma2 eqn1}
\sum_p \sum_{ k+\ell \geq 3 } \frac{ \pi^{k+\ell} (u^2+v^2)^{(k+\ell)/2} }{k! \ell!} | b_{k, \ell} (p^{-1/2} ) - b_{k,\ell}( p^{-\sigma_T})|  \\
\ll   \frac{1}{( \log T)^\theta} \sum_p \sum_{ k+\ell \geq 3 }  
 \bigg(  \frac{ \pi \sqrt{\delta_1}e C_r }{p^{1/2-\epsilon}}\bigg)^{k+\ell}\ll \frac{1}{( \log T)^\theta}
\end{multline}
for $ u^2 + v^2 \leq \delta_1 $, where $\delta_1$ is a constant satisfying $  \frac{ \pi \sqrt{\delta_1}e C_r }{2^{1/2-\epsilon}} < 1$.
By \eqref{log J exp}, we have  
\begin{equation}\label{lemma2 eqn2}
\sum_p \log    J( \pi u , \pi v , p^{- \sigma_T } ) =      \sum_{ k , \ell \geq 1 } \frac{ (\pi i)^{k+\ell}}{k! \ell!} (u+iv)^k (u-iv)^\ell \sum_p  b_{k, \ell} ( p^{-\sigma_T} )  
\end{equation}
 for $u^2 + v^2 \leq \delta_1 $ if $ \delta_1 \leq ( \pi \sqrt{2} C_{1/\sqrt{2}} )^{-2}$. By \eqref{lemma2 eqn1}, \eqref{lemma2 eqn2} and the identity $ \psi_T = \sum_p b_{1,1} (p^{-\sigma_T}) $,  the lemma follows.

\end{proof}

\begin{lemma}\label{lemma3}
There are   constants $\delta_2, \delta_3 >0 $ and a sequence $ \{d_{k,\ell} \}_{k, \ell \geq 0 }$ of real numbers such that 
\begin{equation}\label{eqn lemma3 main}
 \Phrhat(u,v) =    e^{- \pi^2 (u^2+v^2) \psi_T } \bigg( \sum_{ k, \ell \geq 0   } (2 \pi i )^{k+\ell} d_{k,\ell} u^k v^\ell    + O\bigg( \frac{ 1}{ ( \log T)^\theta}   \bigg)  \bigg) 
\end{equation}
for $ u^2 + v^2 \leq \delta_2 $, where $d_{0,0}=1$, $ d_{k,\ell} = 0 $ for $ k+\ell=1,2$ and $d_{k,\ell} = O ( \delta_3^{ -(k+\ell)} ) $ for $ k+\ell \geq 3 $.
\end{lemma}

\begin{proof}
By \eqref{eqn phrhat prod} and Lemma \ref{lemma2}, we have
$$
 \Phrhat(u,v) =  e^{- \pi^2 (u^2+v^2) \psi_T } g(u,v)  \bigg(1  + O\bigg( \frac{ 1}{ ( \log T)^\theta}   \bigg)  \bigg)  $$
for $ u^2 + v^2 \leq \delta_1$, where
$$g(u,v) := \exp \bigg(    \sum_{ \substack{ k , \ell \geq 1 \\ k+\ell \geq 3 }} \tilde{b}_{k , \ell} (u+iv)^k (u-iv)^\ell  \bigg).$$
By \eqref{def tildebkl}, Lemma \ref{lemma1} and Stirling's formula, the sum   
\begin{equation} \label{est doube sum tilde b}\begin{split}
  \sum_{ \substack{ k , \ell \geq 1 \\ k+\ell \geq 3 }} |\tilde{b}_{k , \ell} | | (u+iv)^k (u-iv)^\ell | 
\leq & \sum_{ \substack{ k , \ell \geq 1 \\ k+\ell \geq 3 }} \frac{ (\pi \sqrt{u^2 + v^2  })^{k+\ell}}{k! \ell!} \sum_p   C_{1/\sqrt{2}}^{k+\ell} \min(k,\ell)^{k+\ell } p^{-(k+\ell)/2}  \\
\ll &  \sum_p   \sum_{ \substack{ k , \ell \geq 1 \\ k+\ell \geq 3 }}   \bigg(  \frac{   C_{1/\sqrt{2}}      \pi e  \sqrt{u^2 + v^2  } }{\sqrt{p} }  \bigg)^{k+\ell}      
\end{split} \end{equation}
is convergent and bounded for $ u^2 + v^2 \leq \delta_2 $ provided that $   C_{1/\sqrt{2}}      \pi e  \sqrt{\delta_2  } < \sqrt{2}     $. Thus, we can find a power series expansion of  $g(u,v)$
 for $ u^2 + v^2 \leq \delta_2$. 
 
  Let $ b'_{k, \ell} = \tilde{b}_{k,\ell} ( 2 \pi i )^{-k-\ell}$, then we see that
  \begin{equation}\label{def gxy}
g \bigg( \frac{x}{2 \pi i },   \frac{y}{2 \pi i } \bigg)  = \exp \bigg(    \sum_{ \substack{ k , \ell \geq 1 \\ k+\ell \geq 3 }} b'_{k , \ell} (x+iy)^k (x-iy)^\ell  \bigg).
\end{equation}
 Since 
 $$  b'_{k,\ell}= \frac{ 1}{2^{k+\ell} k! \ell!}\sum_p  b_{k, \ell} ( p^{-1/2} ) $$
 by \eqref{def tildebkl}, we have that $b'_{k,\ell} = b'_{\ell, k}$ and $ b'_{k,\ell}$ is real for every $k, \ell $ by  Lemma \ref{lemma1}. Since
 \begin{align*}
   \sum_{ \substack{ k , \ell \geq 1 \\ k+\ell \geq 3 }}&  b'_{k , \ell} (x+iy)^k (x-iy)^\ell  -   \sum_{ \substack{ k , \ell \geq 1 \\ k+\ell \geq 3 }} \overline{ b'_{k , \ell} (x+iy)^k (x-iy)^\ell } \\& = \sum_{ \substack{ k , \ell \geq 1 \\ k+\ell \geq 3 }} b'_{k , \ell} (x+iy)^k (x-iy)^\ell  -    \sum_{ \substack{ k , \ell \geq 1 \\ k+\ell \geq 3 }}  b'_{k , \ell} (x-iy)^k (x+iy)^\ell    \\
   & = \sum_{ \substack{ k , \ell \geq 1 \\ k+\ell \geq 3 }} b'_{k , \ell} (x+iy)^k (x-iy)^\ell  -    \sum_{ \substack{ k , \ell \geq 1 \\ k+\ell \geq 3 }}  b'_{\ell  , k } (x-iy)^k (x+iy)^\ell    = 0,
   \end{align*}
   the sum $  \sum_{ \substack{ k , \ell \geq 1 \\ k+\ell \geq 3 }} b'_{k , \ell} (x+iy)^k (x-iy)^\ell $ is a power series in $ x $ and $y$ with real coefficients.
Therefore, there is a  sequence $ \{ d_{k,\ell} \}_{k,\ell \geq 0 }$ of real numbers such that
\begin{equation}\label{def dkl}
  \sum_{ k, \ell \geq 0  }   d_{k,\ell} x^k y^\ell := \exp \bigg(    \sum_{ \substack{ k , \ell \geq 1 \\ k+\ell \geq 3 }} b'_{k , \ell} (x+iy)^k (x-iy)^\ell  \bigg) .  
  \end{equation}
By \eqref{def gxy} and \eqref{def dkl}, we have 
$$ g(u,v) =  \sum_{ k, \ell \geq 0  } ( 2 \pi i )^{k+\ell} d_{k,\ell} u^k v^\ell    .$$
This proves \eqref{eqn lemma3 main}.

By expanding the right hand side of \eqref{def dkl}, it is easy to see that $ d_{0,0}=1 $ and $ d_{k,\ell}=0 $ for $k+\ell = 1,2$. 
Let $\delta_3 $ be a constant such that $ 0 <    \delta_3 <  \frac{ \sqrt{2}}{ e C_{1/\sqrt{2}}  }  $. Since $g(u,v)$ is bounded for $ |u| , |v| \leq \frac{\delta_3}{2 \pi} $ similarly to \eqref{est doube sum tilde b},  we have
$$ d_{k,\ell}  = \frac{1}{ ( 2 \pi i )^{ k+ \ell +2 }} \oint_{ |u|= \frac{\delta_3}{2 \pi}} \oint_{ |v|= \frac{\delta_3}{2 \pi} }  \frac{ g(u,v)}{u^{k+1} v^{\ell+1} }  dv  du =  O( \delta_3^{-(k+\ell)} ) . $$ 
 \end{proof}

\begin{lemma}\label{lemma4}
Let $\{ d_{k,\ell} \}_{k,\ell\geq 0 }$ be the sequence of real numbers in Lemma \ref{lemma3}. There exist constants $\epsilon, \eta > 0$ such that 
\begin{align*}
F_{\sigma_T} ( x,y)  &=  e^{ - ( x^2 + y^2 )/\psi_T}   \sum_{ k+ \ell \leq \epsilon \psi_T   } \frac{d_{k,\ell} }{   \pi       \sqrt{\psi_T} ^{k+\ell+2 }}   H_k \bigg(  \frac{x}{\sqrt{\psi_T}}  \bigg)H_\ell \bigg(  \frac{y}{\sqrt{\psi_T}}  \bigg)  +    O \bigg( \frac{ 1}{ ( \log T)^{\eta}} \bigg)
\end{align*}
for all $ x,y \in \mathbb{R}$, where  $H_n (x)$ is the $n$-th Hermite polynomial defined in \eqref{def Hermite poly}. 
\end{lemma}
\begin{proof}
Let $\delta_4 $ be a constant satisfying $ 0< \delta_4 < \min( \delta_2 ,  \delta_3^2 (2 \pi)^{-2})$. By applying Lemma 3.5 of \cite{HL} to  \eqref{F Phrhat int}, there is a constant $\eta_1 > 0 $ such that
$$ F_{\sigma_T } ( x, y ) = \iint_{ u^2+v^2 \leq \delta_4   } \Phrhat( u,v) e^{-  2 \pi i ( ux+vy)} du dv + O \bigg( \frac{1}{ ( \log T)^{\eta_1} } \bigg) .$$
 Let $ \epsilon $ be a constant satisfying $ 0 < \epsilon < \frac{e}{4} \delta_2^2$. By Lemma \ref{lemma3}, we have
\begin{align*}
F_{\sigma_T } ( x, y )   
=  & \sum_{ k, \ell \geq 0   } ( 2\pi i )^{k+\ell}  d_{k,\ell}  \iint_{ u^2+v^2 \leq \delta_4   } e^{- \pi^2 (u^2+v^2) \psi_T }      u^k v^\ell   e^{-  2 \pi i ( ux+vy)}du dv  + O\bigg( \frac{ 1}{ ( \log T)^{\eta_2}}   \bigg) \\ 
=  & \sum_{ k+ \ell \leq \epsilon \psi_T   }  ( 2\pi i )^{k+\ell}  d_{k,\ell}  \iint_{ u^2+v^2 \leq \delta_4   } e^{- \pi^2 (u^2+v^2) \psi_T }      u^k v^\ell   e^{-  2 \pi i ( ux+vy)}du dv \\
&  + O\bigg( \sum_{ k+ \ell > \epsilon \psi_T }   \frac{  ( 2 \pi )^{k+\ell} \delta_4^{(k+\ell)/2}}{\delta_3^{k+\ell} \psi_T}+ \frac{ 1}{ ( \log T)^{\eta_2}}   \bigg)  
\end{align*}
where $\eta_2 = \min( \eta_1 , \theta) $. Since $ \delta_4 <     \delta_3^2 (2 \pi)^{-2} $, the $O$-term is
$  O ( ( \log T)^{-\eta_3} )$ for some $\eta_3 > 0 $.  

To complete the proof, it requires to estimate the last integral, which equals to
\begin{equation}\label{last integral split}
 \iint_{ \mathbb{R}^2  } e^{- \pi^2 (u^2+v^2) \psi_T }   u^k v^\ell    e^{-  2 \pi i ( ux+vy)}du dv  -\iint_{ u^2 + v^2 > \delta_4 } e^{- \pi^2 (u^2+v^2) \psi_T }   u^k v^\ell    e^{-  2 \pi i ( ux+vy)}du dv  .
 \end{equation}
The second integral in \eqref{last integral split} is 
\begin{align*}
\bigg| \iint_{ u^2 + v^2 > \delta_4  } & e^{- \pi^2 (u^2+v^2) \psi_T }   u^k v^\ell    e^{-  2 \pi i ( ux+vy)}du dv  \bigg| \\
&\leq \iint_{ u^2 + v^2 > \delta_4 } e^{- \pi^2 (u^2+v^2) \psi_T }   |u|^k |v|^\ell    du dv \\
&\leq \bigg(  \iint_{ u^2 + v^2 > \delta_4 } e^{- \pi^2 (u^2+v^2) \psi_T }     du dv \bigg)^{1/2}  \bigg( \iint_{ \mathbb{R}^2  } e^{- \pi^2 (u^2+v^2) \psi_T }   u^{2k} v^{2\ell}     du dv \bigg)^{1/2} \\
& =   e^{-  \frac{\pi^2 \delta_4 }{2} \psi_T}  \frac{ \sqrt{\pi} }{ ( \pi \sqrt{\psi_T})^{k+\ell+2}} \sqrt{ \Gamma( k+ \frac12) \Gamma ( \ell+ \frac12 ) } 
\end{align*}
by the Cauchy-Schwartz inequality. By Stirling's formula, the above is
\begin{align*}
& \ll e^{-  \frac{\pi^2 \delta_4}{2} \psi_T} \frac{1}{ ( \pi \sqrt{\psi_T})^{k+\ell+2}} \frac{ ( k+1/2)^{k/2} (\ell+1/2)^{\ell/2}}{ e^{(k+\ell)/2}}  
 \leq  e^{-  \frac{\pi^2 \delta_4}{2} \psi_T} \frac{1}{ \pi^2 \psi_T} \bigg( \frac{ \sqrt{ \epsilon}}{ \pi \sqrt{e}} \bigg)^{k+\ell} 
  \end{align*}
  for $ k+\ell \leq \epsilon \psi_T$. 
Since $ d_{k,\ell} = O( \delta_3^{-(k+\ell)})$ by Lemma \ref{lemma3},    the contribution of the second integral in \eqref{last integral split} to $F_{\sigma_T}(x,y) $ is
$$ O\bigg( \sum_{ k+ \ell \leq \epsilon \psi_T} e^{-  \frac{\pi^2 \delta_4}{2} \psi_T} \frac{1}{  \psi_T} \bigg( \frac{ 2 \sqrt{ \epsilon}}{ \delta_3    \sqrt{e}} \bigg)^{k+\ell} \bigg) = O\bigg( \frac{1}{ ( \log T)^{\eta_4}}\bigg)$$
for some $\eta_4 >0$ since $ \epsilon < \frac{e}{4}      \delta_3 ^2  $. Therefore, we have
$$
F_{\sigma_T } ( x, y )   
 =    \sum_{ k+ \ell \leq \epsilon \psi_T   }  ( 2\pi i )^{k+\ell}  d_{k,\ell}  \iint_{ \mathbb{R}^2 } e^{- \pi^2 (u^2+v^2) \psi_T }      u^k v^\ell   e^{-  2 \pi i ( ux+vy)}du dv 
   + O\bigg( \frac{1}{ ( \log T)^{\eta}}\bigg)  $$
   with $\eta = \min ( \eta_3 , \eta_4)$. Since the last integral equals to  
   \begin{align*}
 &  \frac{1}{ ( -  2 \pi i )^{k+\ell}} \frac{\partial^{k+\ell} }{ \partial x^k \partial y^\ell}    \iint_{ \mathbb{R}^2  } e^{- \pi^2 (u^2+v^2) \psi_T }       e^{-  2 \pi i ( ux+vy)}du dv \\
=&  \frac{1}{ ( -  2 \pi i )^{k+\ell}} \frac{\partial^{k+\ell} }{ \partial x^k \partial y^\ell}\bigg( \frac{1}{ \pi \psi_T} e^{ - ( x^2 + y^2 )/\psi_T}\bigg) \\
= &\frac{1}{ \pi \psi_T} \frac{1}{ (    2 \pi i \sqrt{\psi_T} )^{k+\ell}} e^{ - ( x^2 + y^2 )/\psi_T}  H_k \bigg(  \frac{x}{\sqrt{\psi_T}}  \bigg)H_\ell \bigg(  \frac{y}{\sqrt{\psi_T}}  \bigg) ,
\end{align*}
the lemma holds.

\end{proof}

\begin{proof}[Proof of Theorem \ref{main theorem}]  
  
  The theorem holds by \eqref{eqn clt proof 1} and Lemma \ref{lemma4}.

\end{proof}

 \section{acknowledgemet}
 This work was supported by Incheon National University RIBS Grant in 2020. We thank an anonymous referee for informing us of Hejhal's paper \cite{H}.

\end{document}